\documentclass[a4paper]{amsart}
\usepackage[utf8]{inputenc}
\usepackage[english]{babel}
\usepackage{amsmath}
\usepackage{amsfonts}
\usepackage{amssymb}
\usepackage{amsthm}
\usepackage{array, caption, floatrow, tabularx, makecell, booktabs}
\usepackage{wasysym}
\usepackage{amscdx}
\usepackage{graphicx}
\usepackage{chngcntr}
\usepackage{tikz}
\usepackage{xcolor}
\usepackage{marvosym}
\usepackage{enumitem}
\usepackage{stix}
\usepackage{todonotes}
\usepackage{comment}
\usepackage{theoremref}

\usepackage{hyperref}
\hypersetup{colorlinks=true}

\newenvironment{enumerate-(a)}{\begin{enumerate}[label={\upshape (\alph*)}, leftmargin=2pc]}{\end{enumerate}}
\newenvironment{enumerate-(a)-r}{\begin{enumerate}[label={\upshape (\alph*)}, leftmargin=2pc,resume]}{\end{enumerate}}
\newenvironment{enumerate-(a)-5}{\begin{enumerate}[label={\upshape (\alph*)}, leftmargin=2pc,start=5]}{\end{enumerate}}
\newenvironment{enumerate-(A)}{\begin{enumerate}[label={\upshape (\Alph*)}, leftmargin=2pc]}{\end{enumerate}}
\newenvironment{enumerate-(A)-r}{\begin{enumerate}[label={\upshape (\Alph*)}, leftmargin=2pc,resume]}{\end{enumerate}}
\newenvironment{enumerate-(i)}{\begin{enumerate}[label={\upshape (\roman*)}, leftmargin=2pc]}{\end{enumerate}}
\newenvironment{enumerate-(i)-r}{\begin{enumerate}[label={\upshape (\roman*)}, leftmargin=2pc,resume]}{\end{enumerate}}
\newenvironment{enumerate-(I)}{\begin{enumerate}[label={\upshape (\Roman*)}, leftmargin=2pc]}{\end{enumerate}}
\newenvironment{enumerate-(I)-r}{\begin{enumerate}[label={\upshape (\Roman*)}, leftmargin=2pc,resume]}{\end{enumerate}}
\newenvironment{enumerate-(1)}{\begin{enumerate}[label={\upshape (\arabic*)}, leftmargin=2pc]}{\end{enumerate}}
\newenvironment{enumerate-(1)-r}{\begin{enumerate}[label={\upshape (\arabic*)}, leftmargin=2pc,resume]}{\end{enumerate}}

\newenvironment{enumerate-(star)}{\begin{enumerate}[label={\upshape{(\( \star_{ \arabic*} \))}}, leftmargin=2pc]}{\end{enumerate}}

\makeatletter
\renewcommand\@mkboth[2]{\markboth{#1}{}}
\makeatother

\newtheorem{thm}{Theorem}[section]
\newtheorem{conjecture}{Conjecture}
\newtheorem{lem}[thm]{Lemma}
\newtheorem{prop}[thm]{Proposition}
\newtheorem{claim}{Claim}[thm]
\newtheorem{cor}[thm]{Corollary}

\theoremstyle{definition}
\newtheorem{defn}[thm]{Definition}

\theoremstyle{remark}

\newtheorem{remark}[thm]{Remark}

\newcommand{\upsmallfrown}{\mathbin{\raisebox{0.3ex}{$\smallfrown$}}}

\newcommand{\N}{\mathbb{N}}

\title{Non-Dyadic Laver Algebras}
\date{}

\author[J.P.~Aguilera]{Juan P. Aguilera}
\address{Institute of Discrete Mathematics and Geometry, TU Wien, 1040 Vienna --- Austria}
\email{aguilera@logic.at}

\author[M.~Iannella]{Martina Iannella}
\address{Institute of Discrete Mathematics and Geometry, TU Wien, 1040 Vienna --- Austria}
\email{martina.iannella@tuwien.ac.at}

\thanks{This work was partially supported by the Austrian Science Foundation (FWF) grants 10.55776/ESP3, 10.55776/ESP1625325, 10.55776/STA139, 10.55776/PAT2264325, and 10.55776/EFP2661725.}

\begin{document}
	
\maketitle
	
\begin{abstract}
We study finite \textit{Laver algebras}, the generalizations of Laver tables on sets of cardinality not of the form $2^n$. 
We characterize the conditions under which projections and embeddings between Laver algebras exist and show that there are $2^{\aleph_0}$ pairwise non-isomorphic direct limits of finite Laver algebras. %We also investigate the periods of these algebras through explicit computations, showing that ...}
\end{abstract}
	\setcounter{tocdepth}{1}
\tableofcontents

\section{Introduction}
Let $N = 2^n$. The \(n\)-th \textit{Laver table} $S_{2^n}$ is the finite algebra defined on the set \(\{1, 2,  \dots, N\}\) equipped with a binary operation \(\star\) satisfying the following properties:
\begin{equation}\label{eq:cyclic0}
	a \star 1 = a + 1 \quad \text{ for } a=1,\dots, N-1,\qquad \qquad N \star 1 = 1,
\end{equation}
and for all \(a\leq N, b < N\),
\begin{equation}\label{eq:distributivity0weak}
	a \star (b \star 1) = (a \star b) \star (a \star 1).
\end{equation}

It is not hard to establish the existence of a unique operation $\star$ on $\{1,\hdots, N\}$ satisfying \eqref{eq:cyclic0} and \eqref{eq:distributivity0weak} (see Dehornoy \cite[Lemma 1.2]{De00}). Such an operation will in fact satisfy the full left self-distributivity law:
for all \(a, b, c \leq N\),
\begin{equation}\label{eq:distributivity0}
	a \star (b \star c) = (a \star b) \star (a \star c).
\end{equation}

Laver tables were discovered by Laver \cite{La92}. They are algebraic objects which found applications in various areas of mathematics, including algebra (see Drapal \cite{Dra94,Dra95}), braid theory (see Dehornoy \cite{De94,De95,De00}), probability theory (see Moore \cite{Moo23}), cryptography (see Dehornoy \cite{De06}), and set theory (see Laver \cite{La95}, Laver and Miller \cite{LM11, LM13}, Miller \cite{Mi16}, Dougherty \cite{Do93,Do96}, Dougherty and Jech \cite{DJ97}). They are particularly interesting from the perspective of set theory because the operation of application on rank-to-rank elementary embeddings $j:V_\lambda\to V_\lambda$ is left self-distributive. The existence of such an embedding is not provable in $\mathsf{ZFC}$ and is expressed by the large cardinal axiom I3, one of the strongest axioms currently considered in set theory.

In this article, we consider the more general \textit{Laver algebras} obtained after relaxing the constraint of size $2^n$. These are finite algebras $S_N$ on sets \(\{1, \dots, N\}\) for arbitrary \(N \in \N\), equipped with an operation \(\star\) that satisfies properties \eqref{eq:cyclic0} and  \eqref{eq:distributivity0weak} (see Figure \ref{FigureAlgebra9}). Again, this operation is uniquely determined, although in this general case, the algebras may fail to be left self-distributive
(see \cite[Chapter X]{De00} for an introduction and basic properties of these algebras). At times, we may speak of ``finite Laver algebras'' merely for the purpose of clarity if appropriate even though all Laver algebras are finite as defined here.

\begin{figure}	\renewcommand{\arraystretch}{1.5} % row height factor
	\centering
	\caption{The Laver algebra $S_9$.}

		\begin{tikzpicture}[baseline=(table2.base)]
			\node (table2) {
				\begin{tabular}{c|*{9}{c}}
					\(\mathbf{\star}\)
					& \textbf{1} & \textbf{2} & \textbf{3} & \textbf{4}
					& \textbf{5} & \textbf{6} & \textbf{7} & \textbf{8}
					& \textbf{9}   \\ \hline
					\textbf{1}  & 2 & 4 & 6 & 8 & 9 & 2 & 4 & 6 & 8 \\
					\textbf{2}  & 3 & 4 & 8 & 9 & 3 & 4 & 8 & 9 & 3  \\
					\textbf{3}  & 4 & 9 & 4 & 9 & 4 & 9 & 4 & 9 & 4  \\
					\textbf{4}  & 5 & 6 & 8 & 9 & 5 & 6 & 8 & 9 & 5  \\
					\textbf{5}  & 6 & 9 & 6 & 9 & 6 & 9 & 6 & 9 & 6  \\ 
					\textbf{6}  & 7 & 8 & 9 & 7 & 8 & 9 & 7 & 8 & 9  \\
					\textbf{7} & 8 & 9 & 8 & 9 & 8 & 9 & 8 & 9 & 8  \\
					\textbf{8} & 9 & 9 & 9 & 9 & 9 & 9 & 9 & 9 & 9  \\
					\textbf{9} & 1 & 2 & 3 & 4 & 5 & 6 & 7 & 8 & 9 \\
				\end{tabular}
			};
		\end{tikzpicture}
		\label{FigureAlgebra9}
\end{figure}

We shall be interested in general questions concerning how these algebras differ from the usual Laver tables, as well as in structural connections between algebras $S_n$ and $S_m$, for $n\neq m$. When $n<m$ are powers of $2$, the situation is well understood: the projection 
\[a \mapsto a \mod n\] is a homomorphism from $S_m$ to $S_n$. This leads to an inverse limit $S_{inv}$ which is known to be isomorphic to the free left self-distributive algebra with one generator \cite{DJ97}. Similarly, the \textit{corner embedding} 
\begin{align*}
n-a &\mapsto m-a, &0\leq a < n
\end{align*}
is an embedding from $S_n$ into $S_m$, leading to a direct limit $S_{dir}$ of Laver algebras $S_{2^k}$. See Figure \ref{FigAlgebras48}.

\begin{figure}[H]
	\renewcommand{\arraystretch}{1.5} % row height factor
	\centering
	\caption{The Laver tables $S_4$ and $S_8$.}
\label{FigAlgebras48}
	\begin{minipage}{0.38\textwidth}
		\centering
		% ---- Small 4x4 table ----
		\begin{tikzpicture}[baseline=(table1.base)]
			\node (table1) {
				\begin{tabular}{c|llll}
					\(\mathbf{\star}\) & \textbf{1} & \textbf{2} & \textbf{3} & \textbf{4} \\ \hline
					\textbf{1} & 2 & 4 & 2 & 4 \\
					\textbf{2} & 3 & 4 & 3 & 4 \\
					\textbf{3} & 4 & 4 & 4 & 4 \\
					\textbf{4} & 1 & 2 & 3 & 4 \\
				\end{tabular}
			};
			
			% Red border around the entire small table
		\end{tikzpicture}
	\end{minipage}
	\hspace{-1.8em}
	\begin{minipage}{0.58\textwidth}
		\centering
		% ---- Large 12x12 table ----
		\begin{tikzpicture}[baseline=(table2.base)]
			\node (table2) {
				\begin{tabular}{c|*{12}{c}}
					\(\mathbf{\star}\)
					& \textbf{1} & \textbf{2} & \textbf{3} & \textbf{4}
					& \textbf{5} & \textbf{6} & \textbf{7} & \textbf{8} \\ \hline
					\textbf{1}  & 2 & 4 & 6 & 8 & 2 & 4 & 6 & 8  \\
					\textbf{2}  & 3 & 4 & 7 & 8 & 3 & 4 & 7 & 8  \\
					\textbf{3}  & 4 & 8 & 4 & 8 & 4 & 8 & 4 & 8 \\
					\textbf{4}  & 5 & 6 & 7 & 8 & 5 & 5 & 7 & 8  \\ 
					\textbf{5}  & 6 & 8 & 6 & 8 & 6 & 8 & 6 & 8 \\
					\textbf{6} & 7 & 8 & 7 & 8 & 7 & 8 & 7 & 8  \\
					\textbf{7} & 8 & 8 & 8 & 8 & 8 & 8 & 8 & 8  \\
					\textbf{8} & 1 & 2 & 3 & 4 & 5 & 6 & 7 & 8 \\
				\end{tabular}
			};
		\end{tikzpicture}
	\end{minipage}
\end{figure}

The situation for general Laver algebras is very different. In general, there exists a homomorphism $\pi: S_m\to S_n$ for $n<m$ precisely when $n | m$ and $n$ is a power of $2$ (see \S\ref{sec:projections}; the fact that this condition is sufficient was mentioned in Dehornoy \cite{De00}), so the only inverse limit of finite Laver algebras is  the inverse limit of $\{S_{2^k}: k
\in\mathbb{N}\}$, precisely the structure studied by Laver (see \cite{La92}).

As far as corner embeddings, the situation is much more interesting. In  \S\ref{sec:embeddings} a criterion is given to determine whether such an embedding exists between $S_n$ and $S_m$. As an application of this, we obtain the main result of this article:
\begin{thm}
There exist $2^{\aleph_0}$ pairwise non-isomorphic countably infinite direct limits of finite Laver algebras under corner embeddability.
\end{thm}
In other words, there are $2^{\aleph_0}$ ways of selecting sequences of Laver algebras each of which embeds into the corner of the next one in such a way that no two sequences yield isomorphic structures. The sequence of powers of $2$ is one example.

In \S\ref{sec:periods}, we conclude with some numerical data and comments regarding the  periods of entries on the first rows of Laver algebras.

%%%%%%%%%%%%%%%%%%%%%%%%%%%%%%%%%%%%%%%%%%%%%%%%%%%%%%%%%%%%%%%%%%%%%%%%%

\section{Preliminaries}\label{sec:preliminaries}
\subsection{Facts from number theory}
In working with sequences of Laver algebras, we shall have to make use of some facts concerning modular arithmetic and prime numbers.
To this aim, we recall some well-known theorems from number theory, starting with the Chinese remainder theorem:

\begin{thm}[Chinese reminder theorem]
Let \(n_1, \dots, n_k\) be integers greater than \(1\), and denote by \(n\) their product. If the \(n_i\) are pairwise coprime, and if \(a_1, \dots, a_k\) are any integers, then the system
\begin{align*}
	x & \equiv a_1  \mod n_1\\
& \vdots \\
	x & \equiv a_k  \mod n_k
\end{align*}
has a solution, and any two solutions \(x_1\) and \(x_2\) satisfy \(x_1 \equiv x_2 \mod n\).	
\end{thm}

The second is Dirichlet's theorem on prime numbers in arithmetic progressions:

\begin{thm}[Dirichlet] \label{ThmDirichlet}
	Let \(a\) and \(b\) be coprime. Then there are infinitely many prime numbers of the form \(a+n\ b\).
\end{thm}

\subsection{Laver algebras}
We collect some known properties of Laver algebras that will be needed in the sequel.

\begin{prop}\cite[Proposition 1.6]{De00}\thlabel{thm:power2_LD}
	\begin{enumerate-(i)}
		\item\label{thm:power2_LD-0} For every \(N\), the following relations hold in \(S_N\):
		\[
		a \star b = 
		\begin{cases}
			b & \text{ if } a=N,\\
			a+1 & \text{ if } b=1, \text{ or } a \star (b-1)=N,\\
			> a \star (b-1) & \text{ otherwise.}
		\end{cases}
		\]
		For \(a < N\), there exists \(p \leq N-a\) and \(c_1=a+1< \dots < c_p=N\) such that, for every \(b\), we have \(a \star b=c_i\) with \(i=b\mod p\), hence, in particular, \(a \star b > a\).
		
		\item If $N$ is not a power of\ \( 2 \), there exists no left distributive algebra satisfying the following properties:
		\begin{equation}\label{eq:cyclic}
			a \star 1 = a + 1, \text{ for } a=1,\dots, N-1, \qquad N \star 1 = 1;
		\end{equation}
		
		\item\label{thm:power2_LD-3} For every $N$, there is a unique left distributive algebra with domain \(\{1,\dots,2^N\}\) that satisfies \eqref{eq:cyclic}.
	\end{enumerate-(i)}
\end{prop}
%
%
%Given \(S_N\),
%\begin{itemize}
%	\item the last row is equal to 
%	\[
%	1 \quad 2 \quad \dots \quad N;
%	\]
%	\item the \((N-1)\)-th row is 
%	\[
%	N \quad N \quad \dots; 
%	\] 
%	\item the \((N-2)\)-th row is 
%	\[
%	N-1 \quad N \quad N-1 \quad N \quad \dots.
%	\] 
%\end{itemize} 
%When \(N\) is even, 
%\begin{itemize}
%	\item the \((N-3)\)-th row is
%	\[
%	N-2 \quad N \quad N-2 \quad N \quad \dots;
%	\] 
%	\item the \((N-4)\)-th row is
%	\[
%	N-3 \quad N-2 \quad N-1 \quad N \quad N-3 \quad \dots;
%	\]
%	\item the \((N-5)\)-th row is
%	\subitem if \(N\) divides \(4\), where \(4\) is the period of the \((N-4)\)-th row, 
%	\[
%	N-4 \quad N \quad \dots,
%	\]
%	\subitem otherwise 
%	\[
%	N-4 \quad N-2 \quad N \quad \dots ;
%	\]
%\end{itemize}

\begin{prop}\thlabel{thm:distributivity_power2}\cite[Lemma 1.7]{De00}
	The algebra \(S_N\) is left distributive if and only if the equality \(a \star N=N\) is true for every \(a \leq N\). 
\end{prop}

Applying the previous theorem and Dehornoy \cite[Proposition 1.15]{De00}, one can easily see that Laver tables are left distributive. Another interesting observation, due to Biane \cite{Bi18}, is that the algebra \(S_N\) is a Laver table if and only if property \eqref{eq:distributivity0weak} holds for \(b=N\) as well. The following remark is a slight strengthening of this.

\begin{prop}\thlabel{thm:distributivity_n}
Let \(S_N\) be a Laver algebra and \(a < N\). Then, the equality \(a \star N = N\) holds if and only if property \eqref{eq:distributivity0weak}  is satisfied for \(a\) and \(b=N\). %(where, following the notation of \thref{eq:distributivity0weak}, \(g\) is equal to \(1\)).
\end{prop}

\begin{proof}
	Fix \(a<N\) and suppose that \(a \star N = N\). Then 
	\[
	a \star (N \star 1) = a \star 1 = N \star (a \star 1) = (a \star N) \star (a \star 1).
	\]
	For the converse, suppose that there exists \(a<N\) such that \(a \star N \neq N\), and let \(b=a \star N < N\). Then
	\[
	a \star (N \star 1)= a \star 1 \leq b < b+1= b \star 1 \leq b \star (a +1) = (a \star N) \star (a \star 1), 
	\]
	and hence \(a \star (N \star 1) \neq (a \star N) \star (a \star 1)\).
\end{proof}

Figures \ref{FigureAlgebra9} and \ref{FigAlgebras48} illustrate this proposition: in general, the last column of $S_N$ will have value only $N$ if and only if $N$ is a power of two. Moreover, values different from $N$ are tied to failures of distributivity.

\section{Projections and inverse limits of Laver algebras}\label{sec:projections}

A \textbf{projection} of a structure \((A, \star)\) onto a structure \((A', \star')\) is a surjective homomorphism 
\[ \pi: (A, \star)\to (A', \star'),\] 
i.e. a surjective map \(\pi\colon A \to A'\) such that \(\pi(a \star b)=\pi(a) \star' \pi(b)\) for all \(a,b \in A\). From now on, we abuse notation and use \(\star\) for \(\star'\) as well. 

\begin{thm}\cite[Lemma 1.8]{De00}\thlabel{thm:hom_LD}
	If \(S_N\) is a left distributive algebra and \(N\) divides \(N'\), then the map 
\[a \mapsto (a\mod N)\] 
is a projection of \(S_{N'}\) onto \(S_N\).
\end{thm}

%To the best of our knowledge, there is no explicit statement or proof in the literature of the converse of Theorem~\ref{thm:hom_LD}. 
We therefore prove the converse of Theorem~\ref{thm:hom_LD} here, extending \cite[Proposition~1.6]{De00}, which is restricted to Laver tables, to the full class of Laver algebras.

\begin{thm}\thlabel{thm:proj}
	There is a projection of \(S_{N'}\) onto \(S_N\) if and only if \(N\) divides \(N'\) and \(N\) is a power of \(2\).
\end{thm}

\begin{proof}
	Suppose that \(N\) divides \(N'\) and \(N\) is a power of \(2\). Then by \thref{thm:power2_LD}\ref{thm:power2_LD-3} the algebra \(S_N\) is left distributive, and by applying \thref{thm:hom_LD} we obtain the existence of the desired projection.
	
	For the converse, let \(\pi\) be a projection of \(S_{N'}\) onto \(S_N\). We first claim that \(\pi(1)=1\). Since $\pi$ is surjective, there is $a \in \{1, \dots, N\}$ such that $\pi(a)=1$. If $a=1$ we are done. Suppose that $a>1$, and denote by $b^c = b^{c-1} \star b$ for every $b \in \{1,\dots, N\}$ and $1<c \leq N$. Then
	\[
	1=\pi(a)=\pi(1^a)=\pi(1^{a-1} \star 1)=\pi(1)^{a-1}\star \pi(1),
	\]
	and from the fact that $1$ appears only as an element of the first column of $S_N$, it follows that $\pi(1)=1$. But then it is easy to see by induction that 
	\begin{equation}\label{eq:equality}
		\pi(a)\equiv a \mod N \qquad \text{for all } a \leq N'.
	\end{equation}
	Moreover, we have that \(\pi(N')=\pi(N' \star N')=\pi(N') \star \pi(N')\), and since $N$ is the only element in $S_N$ which is idempotent (this follows by \ref{thm:power2_LD-0} of \thref{thm:power2_LD}), we obtain that \(\pi(N')=N\). From the latter and \eqref{eq:equality} it follows that \(N' = N \cdot k\) for some \(k\). 

	It remains to show that \(N\) is a power of \(2\). Toward a contradiction, suppose that this is not the case. 
	Then by \thref{thm:distributivity_power2} applied to \(S_N\), there is \(a < N \) in \(S_N\) such that \(a \star N \neq N\), and hence by \thref{thm:distributivity_n} we have  
	\[
	a \star (N \star 1) \neq (a\star N) \star (N \star 1).
	\] 
	On the other hand, by property \eqref{eq:distributivity0weak} applied to \(S_{N'}\) it follows that 
	\[
	a \star (N \star 1)= (a \star N) \star (a \star 1),
	\] 
	and thus \(\pi(a \star (N \star 1))=\pi((a \star N) \star (a \star 1))\) holds in \(S_N\), but 
	\[
	\pi(a \star (N \star 1))=\pi(a) \star (\pi(N) \star \pi(1))= a \star (N \star 1)
	\] 
	and 
	\[
	\pi((a \star N) \star (a \star 1))= (\pi(a) \star \pi(N)) \star (\pi(a) \star \pi(1))=(a \star N) \star (a \star 1),
	\] 
	a contradiction. 
\end{proof}

The proof of Theorem \ref{thm:proj} shows:

\begin{cor}
Suppose $\pi: S_{N'} \to S_N$ is a projection. Then, $\pi$ is defined by 
\[\pi(a) = a \mod N.\]
\end{cor}

A consequence of \thref{thm:proj} is the following.

\begin{cor}
	Up to isomorphism, there is a unique countable Laver algebra that arises as an inverse limit of a sequence of finite Laver algebras \(S_N\).
\end{cor} 
\begin{proof}
	By \thref{thm:proj}, any inverse system of Laver algebras must consist of Laver algebras \(S_N\) with \(N=2^k\), with projections compatible via reduction modulo \(2^k\). As a result, the inverse limit is uniquely determined (up to isomorphism) by this system.
\end{proof}

\section{Embeddings and direct limits of Laver algebras}\label{sec:embeddings}

In this section we prove the main result of this article, which concerns direct limits of finite Laver algebras.
Recall that an \textbf{embedding} between two algebras is an injective homomorphism. While in the previous section we have seen that there exist canonical projections of \(S_{N'}\) into \(S_N\) under the conditions of \thref{thm:proj}, we are going to prove that there is no canonical embedding of \(S_N\) into \(S_{N'}\). We first define a very natural embedding that can exist between two Laver algebras.

\begin{defn}
	We say that \(\iota \colon S_N \to S_{N'}\) is a \textbf{corner embedding} if \(\iota\) is an embedding and \(\iota(N-a)=N'-a\), equivalently \(\iota(a)=N'-N+a\), for every \(a<N\).
\end{defn}

Notice that since \(\iota\) is an embedding, we have \(\iota(N)=N'\).

\begin{figure}[H]
	\renewcommand{\arraystretch}{1.5} % row height factor
	\centering
	\caption{Example of the corner embedding \(\iota\) of $S_4$ into $S_{12}$, sending \(4\) to \(12\) and \(4-a\) to \(12-a\) for each \(a \in \{1,2,3\}\).}
	
	\begin{minipage}{0.38\textwidth}
		\centering
		% ---- Small 4x4 table ----
		\begin{tikzpicture}[baseline=(table1.base)]
			\node (table1) {
				\begin{tabular}{c|llll}
					\(\mathbf{\star}\) & \textbf{1} & \textbf{2} & \textbf{3} & \textbf{4} \\ \hline
					\textbf{1} & 2 & 4 & 2 & 4 \\
					\textbf{2} & 3 & 4 & 3 & 4 \\
					\textbf{3} & 4 & 4 & 4 & 4 \\
					\textbf{4} & 1 & 2 & 3 & 4 \\
				\end{tabular}
			};
			
			% Red border around the entire small table
			\draw[red, thick] 
			([xshift=-0.2cm, yshift=0.2cm] table1.south east) rectangle 
			([xshift=-2.2cm, yshift=2.6cm] table1.south east);
		\end{tikzpicture}
	\end{minipage}
	\hspace{-1.8em}
	\begin{minipage}{0.58\textwidth}
		\centering
		% ---- Large 12x12 table ----
		\begin{tikzpicture}[baseline=(table2.base)]
			\node (table2) {
				\begin{tabular}{c|*{12}{c}}
					\(\mathbf{\star}\)
					& \textbf{1} & \textbf{2} & \textbf{3} & \textbf{4}
					& \textbf{5} & \textbf{6} & \textbf{7} & \textbf{8}
					& \textbf{9} & \textbf{10} & \textbf{11} & \textbf{12} \\ \hline
					\textbf{1}  & 2 & 12 & 2 & 12 & 2 & 12 & 2 & 12 & 2 & 12 & 2 & 12 \\
					\textbf{2}  & 3 & 12 & 3 & 12 & 3 & 12 & 3 & 12 & 3 & 12 & 3 & 12 \\
					\textbf{3}  & 4 & 8 & 12 & 4 & 8 & 12 & 4 & 8 & 12 & 4 & 8 & 12 \\
					\textbf{4}  & 5 & 6 & 7 & 8 & 9 & 10 & 11 & 12 & 5 & 6 & 7 & 8 \\
					\textbf{5}  & 6 & 8 & 10 & 12 & 6 & 8 & 10 & 12 & 6 & 8 & 10 & 12 \\
					\textbf{6}  & 7 & 8 & 11 & 12 & 7 & 8 & 11 & 12 & 7 & 8 & 11 & 12 \\
					\textbf{7}  & 8 & 12 & 8 & 12 & 8 & 12 & 8 & 12 & 8 & 12 & 8 & 12 \\
					\textbf{8}  & 9 & 10 & 11 & 12 & 9 & 10 & 11 & 12 & 9 & 10 & 11 & 12 \\ 
					\textbf{9}  & 10 & 12 & 10 & 12 & 10 & 12 & 10 & 12 & 10 & 12 & 10 & 12 \\
					\textbf{10} & 11 & 12 & 11 & 12 & 11 & 12 & 11 & 12 & 11 & 12 & 11 & 12 \\
					\textbf{11} & 12 & 12 & 12 & 12 & 12 & 12 & 12 & 12 & 12 & 12 & 12 & 12 \\
					\textbf{12} & 1 & 2 & 3 & 4 & 5 & 6 & 7 & 8 & 9 & 10 & 11 & 12 \\
				\end{tabular}
			};
			
			% Red square around bottom-right 4x4 subtable
			\draw[red, thick] 
			([xshift=-2.8cm, yshift=0.2cm] table2.south east) rectangle 
			([xshift=-0.2cm, yshift=2.6cm] table2.south east);
		\end{tikzpicture}
	\end{minipage}
	
\end{figure}

Given \(S_N\) and \(a \in S_N\), we recall that the \textbf{period} of \(a\) (equivalently, the period of the \(a\)-th row) in \(S_N\), denoted by \(\tau^N(a)\), is the least integer \(p\leq N\) such that \(a \star p=N\). This is justified by the simple observation that Laver algebras satisfy
\[a \star b = a \star b' \text{ whenever $b \equiv b' \mod \tau^N(a)$},\]
i.e., that the
the entries in a row of a Laver algebras are periodic, with the period always beginning on $a\star 1$ and ending on $N$.

\begin{thm}\thlabel{thm:emb}
	The following are equivalent for \(2<N \leq N'\):
	\begin{enumerate-(i)}
		\item\label{thm:emb-1} There is a corner embedding from \(S_N\) to \(S_{N'}\),
		\item \(N \equiv N' \mod \tau^N(a)\), for all \(a<N\).
	\end{enumerate-(i)}
	In particular, if \(N \equiv N' \mod p_j^{e_j}\) for every power \(p_j^{e_j}\) of a prime \(p_j\) such that \(p_j^{e_j} < N\), then condition \ref{thm:emb-1} holds, and hence for each \(n \in \N\), there is a corner embedding of \(S_N\) into \(S_{N + n \prod \{p_j^{e_j} \mid p_j^{e_j} < N\}}\).
\end{thm}

\begin{proof}
	Assume the existence of a corner embedding and denote it by \( \iota \). Fix \( a < N \). Observe that \( \tau^N(a) = \tau^{N'}(\iota(a)) \). 
	Then the first element of the \( \iota(a) \)-th row of \( S_{N'} \) that lies in the range of \( \iota \) is given by
	\[
	%(N' - N + a) \star_{N'} (N' - N + 1) = 
	\iota(a) \star \iota(1) = \iota(a \star 1) = \iota(a + 1) = N' - N + a + 1.
	\]
	
	Since the first element of the \( \iota(a) \)-th row of \( S_{N'} \) is also \( N' - N + a + 1 \), it follows that the number of elements in the \( \iota(a) \)-th row which are not in the range of \( \iota \), and which equals \( N - N' \), must be a multiple of \( \tau^{N'}(\iota(a)) = \tau^N(a) \). 
	Therefore, we conclude that
	\[
	N \equiv N' \mod \tau^N(a).
	\]
	Conversely, suppose that \(N \equiv N' \mod \tau^N(a)\) for all \(a<N\), and let $\iota$ be the map defined by \(\iota(a)=N'-N+a\) for every \(a \leq N\). We want to show that \(\iota\) is an embedding of \(S_N\) into \(S_{N'}\), i.e. that for all \(a,b \in S_N\), 
	\begin{equation}\label{eq:thm_emb}
		%(N'-N+a) \star (N'-N+b) = 
		\iota(a) \star \iota(b) = \iota(a \star b) .
		%= N'-N + (a \star b).
	\end{equation}
We prove this by descending induction on \(a\). Fix \(a < N\) and assume that \eqref{eq:thm_emb} is true for all \(c\) such that \(a<c<N\). Let us look at the \(\iota(a)\)-th row of \(S_{N'}\). The first element is \(\iota(a)+1=N'-N+a+1\). The second one is
\begin{align*}
	\iota(a) \star 2 & = (\iota(a)\star 1)) \star (\iota(a)\star 1)= (N'-N+a+1) \star (N'-N+a+1)\\
	& = (N'-N+(a+1)) \star (N'-N+(a+1))=\iota(a+1) \star \iota(a+1),
\end{align*}
and by applying the inductive hypothesis to \(a+ 1\), from the latter we obtain
\[
= \iota((a+1)\star (a+1))=\iota(a \star 2).
\]
Then it is easy to see that for each \(2 \leq b \leq N\), we have 
\[
\iota(a) \star b = \iota(a \star b).
\]
In particular, \(\tau^{N'}(\iota(a))=\tau^N(a)\).
Now, from the hypothesis that \(N \equiv N' \mod \tau^N(a)\), it follows that
\[
\iota(a) \star b = \iota(a) \star (N'-N+b),
\]
hence 
\[
\iota(a) \star \iota(b) = \iota(a) \star (N'-N+b) = \iota(a) \star b = \iota(a \star b).
\] 

\end{proof}

\begin{remark}
	There are embeddings between Laver algebras which are not corner embeddings. An example is illustrated in Figure \ref{FigureNonCorner}.
\end{remark}
	
\begin{figure}[H]
	\centering
	\caption{Embedding of \(S_4\) into \(S_{12}\), sending \(1\) to \(7\), \(2\) to \(8\), \(3\) to \(11\), and \(4\) to \(12\).}
\label{FigureNonCorner}	{\renewcommand{\arraystretch}{1.5} % row height factor
		
		\begin{minipage}{0.38\textwidth}
			\centering
			% ---- Small 4x4 table ----
			\begin{tikzpicture}[baseline=(tableSmall.base)]
				\node (tableSmall) {
					\begin{tabular}{c|llll}
						\(\mathbf{\star}\) & \textbf{1} & \textbf{2} & \textbf{3} & \textbf{4} \\ \hline
						\textbf{1} & 2 & 4 & 2 & 4 \\
						\textbf{2} & 3 & 4 & 3 & 4 \\
						\textbf{3} & 4 & 4 & 4 & 4 \\
						\textbf{4} & 1 & 2 & 3 & 4 \\
					\end{tabular}
				};
				
				% Contorno rosso attorno alla 4x4
				\draw[red, thick] 
				([xshift=-0.2cm, yshift=0.2cm] tableSmall.south east) rectangle 
				([xshift=-2.2cm, yshift=2.6cm] tableSmall.south east);
			\end{tikzpicture}
		\end{minipage}
		\hspace{-1.8em}
		\begin{minipage}{0.58\textwidth}
			\centering
			% ---- Large 12x12 table ----
			\begin{tikzpicture}[baseline=(table.base)]
				\node (table) {
					\begin{tabular}{c|*{12}{c}}
						\(\mathbf{\star}\)
						& \textbf{1} & \textbf{2} & \textbf{3} & \textbf{4}
						& \textbf{5} & \textbf{6} & \textbf{7} & \textbf{8}
						& \textbf{9} & \textbf{10} & \textbf{11} & \textbf{12} \\ \hline
						\textbf{1}  & 2 & 12 & 2 & 12 & 2 & 12 & 2 & 12 & 2 & 12 & 2 & 12 \\
						\textbf{2}  & 3 & 12 & 3 & 12 & 3 & 12 & 3 & 12 & 3 & 12 & 3 & 12 \\
						\textbf{3}  & 4 & 8 & 12 & 4 & 8 & 12 & 4 & 8 & 12 & 4 & 8 & 12 \\
						\textbf{4}  & 5 & 6 & 7 & 8 & 9 & 10 & 11 & 12 & 5 & 6 & 7 & 8 \\
						\textbf{5}  & 6 & 8 & 10 & 12 & 6 & 8 & 10 & 12 & 6 & 8 & 10 & 12 \\
						\textbf{6}  & 7 & 8 & 11 & 12 & 7 & 8 & 11 & 12 & 7 & 8 & 11 & 12 \\
						\textbf{7}  & 8 & 12 & 8 & 12 & 8 & 12 & 8 & 12 & 8 & 12 & 8 & 12 \\
						\textbf{8}  & 9 & 10 & 11 & 12 & 9 & 10 & 11 & 12 & 9 & 10 & 11 & 12 \\ 
						\textbf{9}  & 10 & 12 & 10 & 12 & 10 & 12 & 10 & 12 & 10 & 12 & 10 & 12 \\
						\textbf{10} & 11 & 12 & 11 & 12 & 11 & 12 & 11 & 12 & 11 & 12 & 11 & 12 \\
						\textbf{11} & 12 & 12 & 12 & 12 & 12 & 12 & 12 & 12 & 12 & 12 & 12 & 12 \\
						\textbf{12} & 1 & 2 & 3 & 4 & 5 & 6 & 7 & 8 & 9 & 10 & 11 & 12 \\
					\end{tabular}
				};
				
				% ---- Red rectangles around 2x2 blocks ----
				\draw[red, thick] ([xshift=-4.2cm, yshift=0.2cm] table.south east) rectangle ([xshift=-3.0cm, yshift=1.4cm] table.south east);
				\draw[red, thick] ([xshift=-1.4cm, yshift=0.2cm] table.south east) rectangle ([xshift=-0.2cm, yshift=1.4cm] table.south east);
				\draw[red, thick] ([xshift=-3.0cm, yshift=2.7cm] table.south east) rectangle ([xshift=-4.2cm, yshift=3.9cm] table.south east);
				\draw[red, thick] ([xshift=-0.2cm, yshift=2.7cm] table.south east) rectangle ([xshift=-1.4cm, yshift=3.9cm] table.south east);
			\end{tikzpicture}
		\end{minipage}
	}
\end{figure}

	%In other words, there exist \(N'\) such that \(S_N\) is not an end-segment of \(\{1,\dots,N'\}\) and \(S_N\) is isomorphic to some \(S_{N'}\).

According to Theorem \ref{thm:emb}, every Laver algebra can be corner-embedded into a larger Laver algebra. Our next goal is to show that this can be done in many different ways, resulting in a family of $2^{\aleph_0}$ non-isomorphic direct limits.

\begin{defn}
	Let \(x: \mathbb{N}\to\mathbb{N}\) be an infinite sequence of natural numbers. We say that \(x\) is \textbf{Laver-suitable} if for each \(i\in\mathbb{N}\), we have
	\[x(k+1)
	\equiv x(k) \mod \tau^{x(k)}(j) \text{ for all $j<x(k)$}.\]
\end{defn}

If \(x\) is Laver-suitable, then it gives rise to an infinite algebra \(S_x\) obtained as the direct limit of \(\{S_{x(k)}: k\in\mathbb{N}\}\) under corner embeddability.
The algebra \(S_x\) where \(x(k) = 2^k\) has been studied before; we shall argue that there are many other algebras of this kind.  We need the following lemma:

\begin{lem}\thlabel{lem:emb_primes}
	For each prime number \(N\) and each \(k<N\), there is a prime number \(N'\) such that \(S_N\) corner embeds into \(S_{N'}\) and \(N' = k \mod N\).
\end{lem}
\begin{proof}
		First, notice that for \(a < N\), \(\tau^N(a)\) does not divide \(N\) and indeed if \(N \equiv k \mod \tau^N(a)\), then either \(\tau^N(a) \not\equiv 0 \mod k\) or \(k=1\). To see this, it suffices to note that otherwise 
		\[N=\tau^N(a)n+k=k m n+k=k(mn+1)\] 
		for some integers \(n\) and \(m\), against the assumption that \(N\) is prime).
	
	Let now \(p_0^{e_0}, \dots , p_\ell^{e_\ell}\) be all the powers of prime numbers occurring in the factorizations of the periods \(\tau^N(a)\), for \(a<N\), with the largest exponent indicated for each prime. So \(p_0^{e_0}, \dots , p_\ell^{e_\ell}\) are pairwise coprime. By \thref{thm:emb}, for \(N<N'\), whether \(S_N\) embeds into \(S_{N'}\) depends only on \(\{N' \mod \tau^N(a) \mid a <N\}\), and thus only on \(\{N' \mod p_i^{e_i} \mid i \leq \ell\}\). So we have 
    %a system of congruences 
	%\begin{equation}
	%	N'\equiv k_i \mod p_i^{e_i}, \quad (i\leq \ell)
	%\end{equation}
	%such 
    that for \(N<N'\), \(S_N\) embeds into \(S_{N'}\) if and only if \(N'\) satisfies the congruences 
	\begin{equation}\label{eq:congruences}
		x \equiv N \mod p_i^{e_i}, \quad (i \leq \ell).
	\end{equation}
Writing
	\[P = N\cdot\prod_{0\leq i \leq \ell} p_i^{e_i},\] 
it follows from the Chinese reminder theorem that there is a unique natural number \(n < P\) %where \(P= p_0^{e_0} \cdot \dots \cdot p_\ell^{e_\ell}\cdot N\) denotes the product of \(p_0^{e_0}, \dots , p_\ell^{e_\ell},N\), 
such that 
	\[
	x \equiv n \mod P \iff x \text{ satisfies } \eqref{eq:congruences} \text{ and } x \equiv k \mod N.
	\]  
	
	\begin{claim}
 \(n\) does not divide \(P\).
	\end{claim}
	
	\begin{proof}[Proof of the claim.]
		Suppose it does. As $N$ is prime, $n$ does not divide $N$. 
	Then \(n = p_i^b\) for some \(i \leq \ell\) and \(b \leq e_i\). But since \(n\) satisfies \eqref{eq:congruences}, we have 
		\[
		N \equiv p_i^b \mod p_i^{e_i},
		\]
		i.e., \(N=p_i^{e_i} m + p_i^b\) for some natural number \(m\), contradicting the fact that \(p_i\) does not divide \(N\).
	\end{proof}
	
	%Notice also that \(n\) does not divide \(N\), as \(N\) is prime, so we conclude that \(n\) does not divide \(P\). 
	
	By Dirichlet's theorem \ref{ThmDirichlet}, there are infinitely many prime numbers \(N'\) of the form 
	\[
	n + m\ P.
	\]
	It is easy to see that \(N'\) satisfies all the congruences above, and hence by \thref{thm:emb}, \(S_N\) embeds into \(S_{N'}\).
\end{proof}

\begin{lem}\thlabel{lem2:emb_primes}
	Let \(N\) be a prime number. Then there are prime numbers \(N'\) and \(N''\), with \(N < N' < N''\), such that the following hold:
	\begin{enumerate-(i)}
		\item There is a corner embedding of \(S_N\) into \(S_{N'}\),
		\item There is a corner embedding of \(S_N\) into \(S_{N''}\),
		\item There is no corner embedding of \(S_{N'}\) into \(S_{N''}\).
	\end{enumerate-(i)} 
\end{lem}
\begin{proof}
	Let \(k\) and \(k'\) be two different numbers \(<N\). Then by \thref{lem:emb_primes} there are prime numbers \(N'\) and \(N''\) such that \(S_N\) corner embeds into both \(S_{N'}\) and \(S_{N''}\), with \(N' = k \mod N\) and \(N'' = k' \mod N\). We now show that \(S_{N'}\) does not embed into \(S_{N''}\).
	
	Notice that \(N'-N''=k+n N - k' - n'N=(k-k')+(n-n')N\), for some integers \(n\) and \(n'\). Since \(k'-k'' \neq 0\), we have that \(N' \not\equiv N'' \mod N\). We now show that \(N\) is the period of the \((N'-N)\)-th row of \(S_{N'}\), so by \thref{thm:emb} there is no embedding of \(S_{N'}\) into \(S_{N''}\).
	
	We show by induction that for \(a\leq N\), the \(a\)-th element of the \((N'-N)\)-th row is equal to \(N'-N+a\), whence it follows that \(\tau^{N'}(N'-N)=N\). The first element of the \((N'-N)\)-th row is given by \(N'-N+1\). Let now assume that the statement is true for \(a < N\) and let us show it for \(a+1\). We have:
	\[
	(N'-N)\star(a+1)=((N'-N)\star a) \star ((N'-N)\star 1)= (N'-N+a)\star (N'-N+1),
	\]	
	and by the existence of the corner embedding \(\iota\) of \(S_N\) into \(S_{N'}\) we obtain
	\[(N'-N+a)\star (N'-N+1)=\iota(a) \star \iota(1)=\iota(a\star 1)=\iota(a+1)=N'-N+(a+1),\]
	as desired.
\end{proof}

We are now ready to establish the main result of this section.
	
\begin{thm}\label{thm:inv_lim}
	There exist $2^{\aleph_0}$ pairwise non-isomorphic direct limits of finite Laver algebras under corner embeddability.
\end{thm}
\begin{proof}
We form a perfect set of direct limits of finite Laver algebras, i.e., an injection of the Cantor set into the set of isomorphism classes of direct limits of finite Laver algebras under corner embeddability. Define a function \(f\) from \(2^{<\mathbb{N}}\) to the set of prime numbers by induction. If \(f(s) = N\), with \(N\) a prime number, then we define \(f(s\upsmallfrown0)\) and \(f(s\upsmallfrown1)\) to be, respectively, the prime numbers \(N'\) and \(N''\) given by \thref{lem2:emb_primes}. Then for each element \(x\) of the Cantor space we have a different Laver algebra obtained as the direct limit of \(S_N\) for \(N\) of the form \(f(x \restriction k)\). If $x$ and $y$ are two distinct elements of the Cantor set $2^{\mathbb{N}}$, then, letting $k$ be least such that $x(k+1) \neq y(k+1)$, say $x(k) = y(k) = N, x(k+1) = N', y(k+1) = N''$, \thref{lem2:emb_primes} itself yields the non-existence of corner embeddings between $S_{N'}$ and $S_{N''}$. This implies that the direct limits $S_x$ and $S_y$ cannot be isomorphic.
\end{proof}

\section{Periods of Laver algebras}\label{sec:periods}
We would like to conclude with some comments on the periods of Laver algebras.

It is commonly known that
\[\lim_{n\to\infty} \tau^{2^n}(1) = \infty,\]
i.e., that the periods of the first rows of Laver tables grow unboundedly. However, all known proofs of this fact require the use of the extremely strong large cardinal axiom $I3$ asserting the existence of a non-trivial elementary embedding 
\[j:V_\lambda\to V_\lambda\]
from a rank-initial segment of the set-theoretic universe to itself (see \cite{La92,La95,DJ97}). By work of Dougherty and Jech \cite{DJ97}, this result is not provable in Primitive Recursive Arithmetic, and this is due to the fact that the function 
\[k \mapsto \text{ least $n$ such that $\tau^{2^n}(a) \geq k$}\]
is extremely fast growing. Indeed, Dougherty \cite{Do93} showed that this function dominates the Ackermann function and Qi \cite{Qi26} has recently showed that it dominates all the provably total recursive functions of Peano Arithmetic.

We would like to finish with some numerical observations on the values of $\tau^n(1)$ for various $n$. Figure \ref{period1} below plots function $n\mapsto \tau^n(1)$, for $n\leq 125,000$. Figure \ref{period2} plots the frequency with which each period appears. We make a few brief observations:

\begin{enumerate}
\item Unlike the situation for powers of $2$, the function $n\mapsto \tau^n(1)$ is not strictly increasing.
\item The function $n\mapsto \tau^n(1)$ appears to grow much faster for arbitrary $n$ than for powers of $2$. Indeed, let $2^n$ be least such that $\tau^{2^n}(1) > 16$. 
Qi \cite{Qi26} has given an extremely large lower bound for $n$: letting $m_0(a) = a$ and $m_{\alpha+1}(a) = m_\alpha^a(a)$ and, for $\alpha = \omega (=\mathbb{N})$ letting $m_\omega(a) = m_a(a)$, he has shown that
\[2^n < m_{\omega+2}(2).\]
On the other hand, $\tau^{280}(1) = 28$, $\tau^{3072}(1) = 32$, and $\tau^{59944}(1) = 36$ already.
\end{enumerate}

 \begin{figure}[ht]
 	\centering \includegraphics[width=0.9\textwidth]{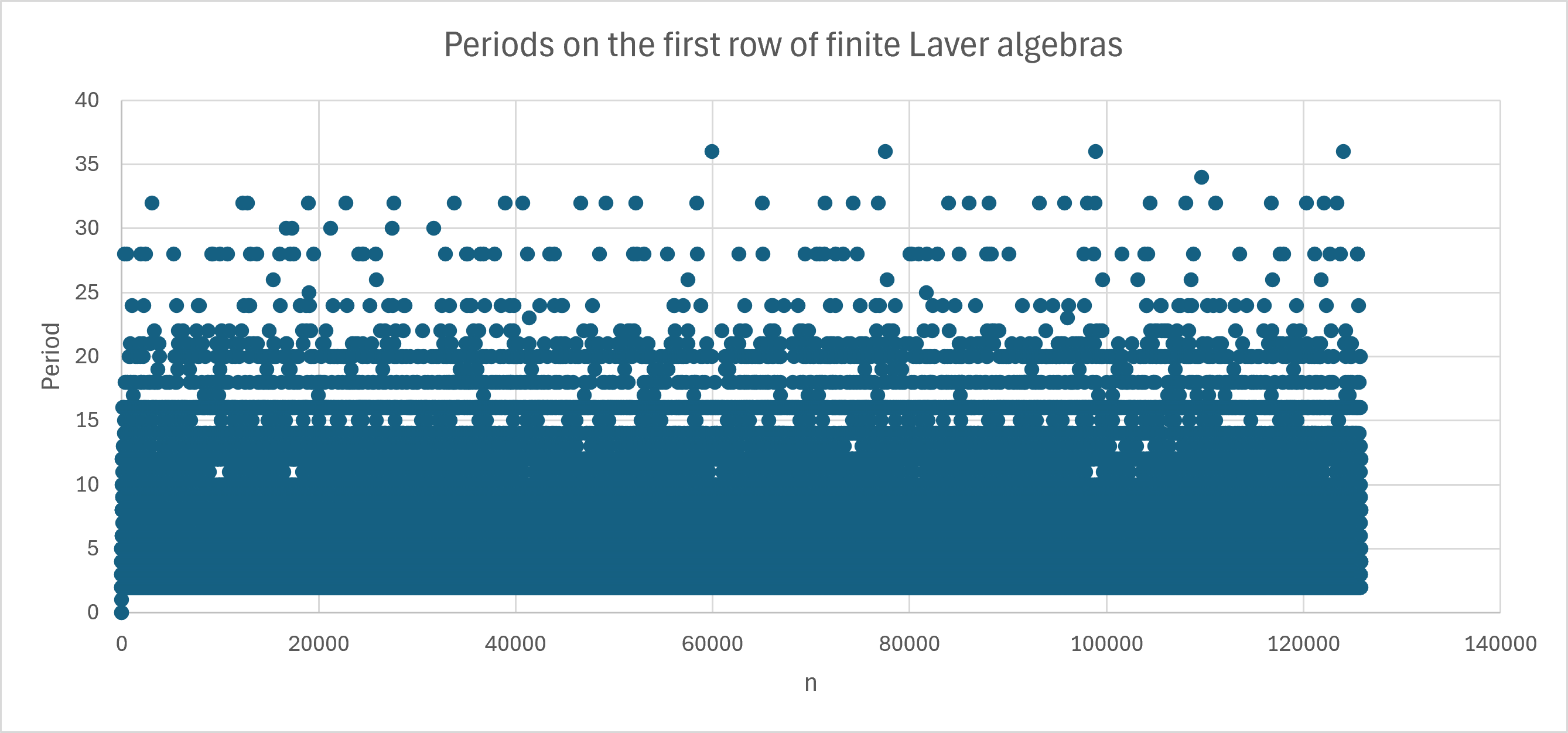}
 	\caption{Scatter plot of values of $\tau^n(1)$ for $n$ below $\approx 125,000$.}
 	\label{period1}
 \end{figure}

 \begin{figure}[ht]
	\centering \includegraphics[width=0.9\textwidth]{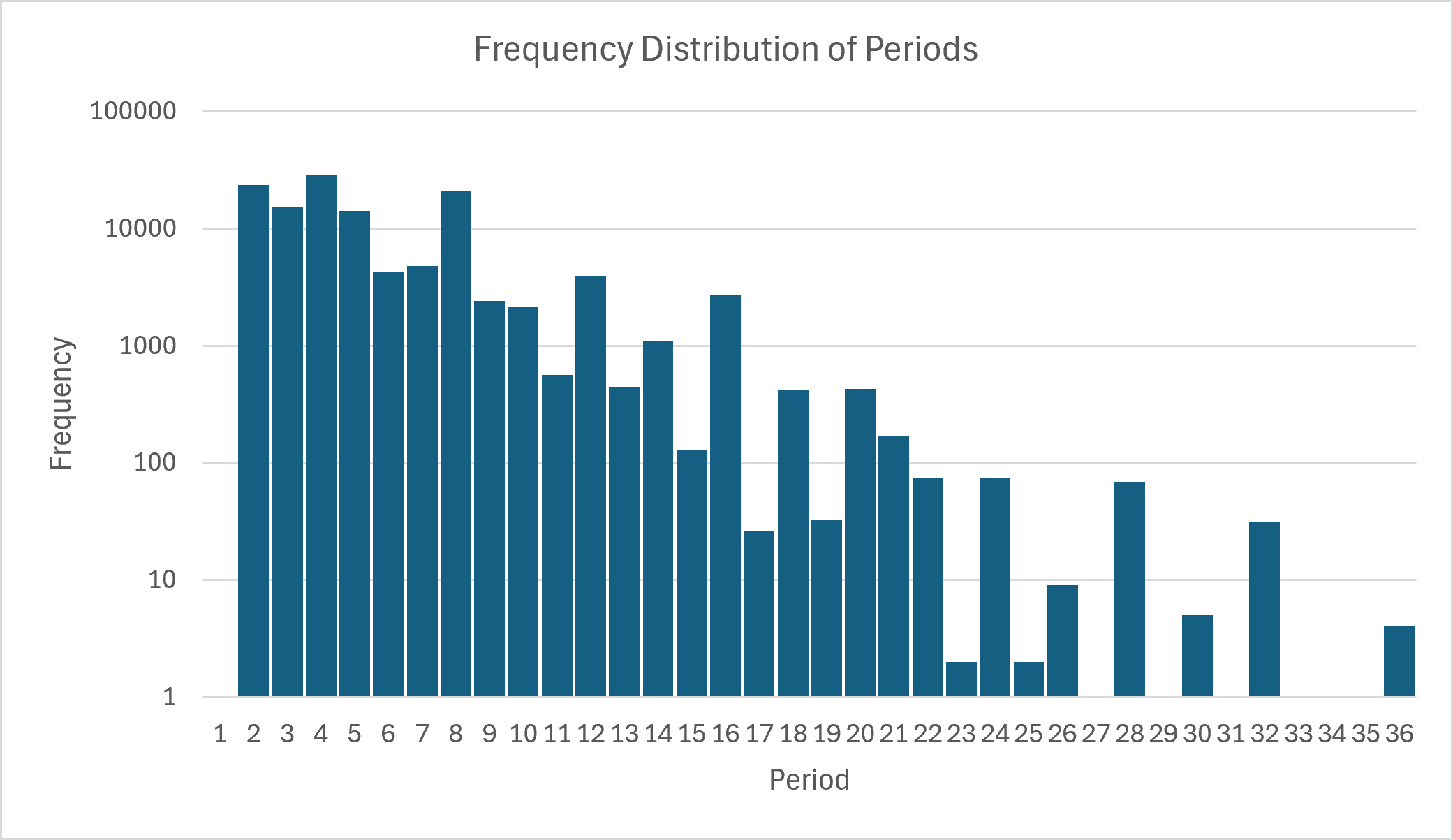}
	\caption{Frequency plot of values of $\tau^n(1)$ for $n$ below $\approx 125,000$.}
	\label{period2}
\end{figure}

We make the following conjectures:
\begin{conjecture}
It is provable in Peano Arithmetic that $\lim_{n\to\infty}(\tau^n(1)) = \infty$.
\end{conjecture}

\begin{conjecture}
For every $a \geq 2$, there is $n$ such that $\tau^n(1) = a$. Indeed, there are infinitely many such $n$.
\end{conjecture}

\bibliographystyle{alpha}
\bibliography{bibliography}
	
\end{document}